\documentclass{amsart}
\usepackage{hyperref,enumerate}
\usepackage{graphicx}
\usepackage{color}

\newcommand{\bg}{\bar{g}}

\newtheorem{theorem}{Theorem}[section]
\newtheorem{lemma}[theorem]{Lemma}
\newtheorem{claim}[theorem]{Claim}
\newtheorem{proposition}[theorem]{Proposition}
\newtheorem{corollary}[theorem]{Corollary}

\theoremstyle{definition}

\newtheorem{remark}[theorem]{Remark}
\numberwithin{equation}{section}


\title[Mass and Volume]{Mass and volume of four-dimensional \\ Einstein metrics }

\author{Matthew J. Gursky}
\address{Department of Mathematics \\
         University of Notre Dame\\
         Notre Dame, IN 46556}
\email{\href{mgursky@nd.edu}{mgursky@nd.edu}}

\author{Andrea Malchiodi}
\address{Classe di Scienze \\
         Scuola Normale Superiore \\
         Piazza dei Cavalieri 7, 56126 \\
         Pisa, Italy}
\email{\href{andrea.malchiodi@sns.it}{andrea.malchiodi@sns.it}}

\begin{document}

\maketitle

\begin{abstract}  Let $(M^4,\bg)$ be an Einstein manifold, where $M^4$ is a smooth, closed, oriented four-manifold $M^4$ and $\bg$ has positive Einstein constant. Given a point $0 \in M^4$, let $G$ denote 
the (positive) Green's function $G$ of the conformal laplacian $L_{\bg}$; then $g = G^2\bg$ is a complete, scalar-flat, asymptotically flat metric on $\widehat{M} = M \setminus \{ 0 \}$.  We first show that the ADM mass of $g$ can 
be expressed as an integral over $\widehat{M}$, then use this identity to prove a lower bound for the mass of $g$ in terms of the volume of $\bg$.   As corollaries, we prove a 'mass times volume' inequality, plus various mass gap theorems 
characterizing the round metric on $S^4$ and the Fubini-Study metric on $\mathbb{CP}^2$.  
\end{abstract}

\section{Introduction and Statement of Main results}

A long open question is whether the round metric $g_0$ is the unique Einstein metric on $S^4$.  For higher dimensional spheres there are many examples of non-round 
Einstein metrics (e.g., \cite{Bohm}, \cite{BGK}, \cite{CS}, \cite{Jen}, \cite{BK}).   

If $\bg$ is a {\em positive} Einstein metric on $S^4$ normalized so that $Ric(\bg) = 3 \bg$, then the Bishop volume comparison theorem implies that 
\begin{align} \label{Bishop} 
Vol(\bg) \leq Vol(g_0) = \frac{8}{3}\pi^2.
\end{align}
Moreover, by Cheng's diameter rigidity result \cite{Cheng}, equality holds if and only if $\bg$ is isometric to $g_0$.  A fairly straightforward second variation calculation for the normalized Einstein-Hilbert action shows that $g_0$ is isolated as an Einstein metric (see \cite{Besse}, Chapter 12H), hence for a non-round Einstein metric $\bg$ we have 
\begin{align} \label{BVC} 
Vol(\bg) \leq \left(1 - \epsilon_0 \right) Vol(g_0)
\end{align} 
for some $\epsilon_0 > 0$.   
 
In \cite{Gur}, the first author made a quantitative improvement of this 'volume rigidity' statement: in fact, for $\bg$ non-round, we must have 
\begin{align} \label{GVC}
Vol(\bg) < \frac{1}{3} Vol(g_0). 
\end{align}
A similar gap result holds for the volume any Einstein metric on $\mathbb{CP}^2$, in this case relative to the volume of $g_{FS}$, the Fubini-Study metric. 
Akutagawa-Endo-Seshadri \cite{AES} subsequently showed that the constant $1/3$ in (\ref{GVC}) is not sharp: there is some $\epsilon_1 > 0$ such that 
\begin{align} \label{GVC}
Vol(\bg) < \left( \frac{1}{3} - \epsilon_1 \right) Vol(g_0). 
\end{align}

Without additional assumptions it seems very difficult to give {\em lower} bounds on the volume of a Einstein metric.  For example, on $S^5$ there are Einstein metrics $\bg$ with $Ric(\bg) = 4 \bg$ and arbitrarily small volume \cite{CS}.  In this article we study another invariant associated to positive Einstein metrics (the ADM mass), and prove a series of gap estimates.   As we will see, these estimates allow for the metric to have arbitrarily small volume.  

In the following we let $M^4$ be a smooth, closed, oriented four-manifold, and $\bg$ a positive Einstein metric on $M^4$ normalized so that $Ric(\bg) = 3 \bg$.  It follows 
that the scalar curvature and the conformal laplacian satisfy $\overline{R} = 12$ and  
\begin{align*}
L = \Delta_{\bg} - \frac{1}{6}\overline{R} = \Delta_{\bg} - 2. 
\end{align*}
Fix a point $0 \in M$ and let $G$ denote the Green's function for $L$ with pole at $0$.  We normalize $G$ so that near $0$,
\begin{align} \label{Gnorm}
G(x) = \dfrac{1}{r(x)^2} + O(1), 
\end{align}
where $r(x)$ is the distance to $0$.  Since $G > 0$ we can define the metric $g = G^2 \bg$, which is complete and scalar-flat metric on $\widehat{M} = M \setminus \{ 0 \}$.  In fact, $(\widehat{M}, g)$ is asymptotically flat of order two (see \cite{LP}, Theorem 6.5): i.e., there is a compact set $K \subset \widehat{M}$ and coordinates $\{ x^i \}$ on $\widehat{M} \setminus K$ such that 
\begin{align*}  
g_{ij} = \delta_{ij} + O(|x|^{-2}), \ \ \partial_k g_{ij} = O(|x|^{-3}), \ \ \partial_k \partial_{\ell} g_{ij} = O(|x|^{-4}). 
\end{align*}
It follows that the {\em ADM-mass} $m_{ADM}$ of $g$ is well defined (\cite{Bartnik}): 
\begin{align} \label{mADM} 
m_{ADM}(g) = \lim_{\rho \to \infty} \frac{1}{\omega_3} \oint_{|x| = \rho} \sum_{i,j} \left( \partial_i g_{ij} - \partial_j g_{ii} \right) \iota_{\partial_j} dv_g.
\end{align}
Although $g$ (hence $m_{ADM}(g)$) depends on the choice of the point $0 \in M$, we will suppress this dependence.  

In part what makes the mass so elusive is that it is given by the limit of a kind of flux integral.  Our first result shows that when $\bg$ is Einstein, the mass can be expressed as an integral over $M \setminus \{ 0 \}$: 

\begin{theorem} \label{Thm1}  Let $(M^4, \bg)$ be an Einstein manifold with $Ric(\bg) = 3 \bg$.  Fix a point $0 \in M^4$, and let $G$ denote the Green's function for $L = \Delta_{\bg} - 2$ with pole at $0$.   Let $g = G^2 \bg$.   Then the mass $m(g) = m_{ADM}(g)$ is given by 
\begin{align} \label{MV1}
 16 \pi^2 m(g) = 6 \left( \frac{8}{3} \pi^2 - \overline{V} \right) + 2 \int_{\widehat{M}} |\mathring{\nabla}_g^2 G|^2 \dfrac{|\nabla G|^2}{G^2} \, dv_g + \int_{\widehat{M}} |\nabla_g F|^2 \, dv_g,
\end{align}
where $\mathring{\nabla}^2_g$ denotes the trace-free Hessian of $G$ (with respect to $g$), $\overline{V} = Vol(\bg)$, and 
\begin{align*}
F = \dfrac{|\nabla_g G|^2 + 1 }{G} - 4. 
\end{align*}
\end{theorem}
Although there may be formulas analogous to \eqref{MV1} in dimensions greater than four, the expansion of the Green's function in Riemannian normal coordinates (see Section \ref{SecG}), as well as in the form of the function $F$ above, seem somewhat special to four dimensions.

In \cite{HL}, Hein-LeBrun gave an explicit formula for the mass of K\"ahler ALE manifolds in terms of topological invariants.  In Theorem \ref{Thm3} below we use (\ref{MV1}) to 
give a lower bound for the mass in terms of topological data, but due to the nature of our estimates we only expect equality in the case of the round metric.  

By (\ref{Bishop}) each term on the right-hand side of (\ref{MV1}) is non-negative, and in case $m_{ADM}(g) = 0$ then we have equality in (\ref{Bishop}) and $(M^4,\bg)$ is isometric to the round sphere.   Of course, by the Positive Mass Theorem (\cite{SY1}, \cite{SY2}, \cite{Witten}), $m_{ADM}(g) \geq 0$ with equality if and only if $(\widehat{M}, g)$ is isometric to $\mathbb{R}^4$ with the flat metric.  In our case, since $g$ arises from the Green's function, equality holds if and only if $(M^4,\bg)$ is isometric to the round metric.  Using the identity (\ref{MV1}), we can prove a lower bound for the mass in terms of the volume gap in (\ref{Bishop}): 

\begin{theorem} \label{Thm2}  Under the assumptions of Theorem \ref{Thm1}, the mass satisfies 
\begin{align} \label{massgap}
8 \pi^2 m(g) \, \overline{V} \geq 12 \left( \frac{8}{3} \pi^2 - \overline{V} \right) \left[ \frac{2}{3} \pi^2 + \left( \frac{8}{3} \pi^2 - \overline{V} \right) \right]. 
\end{align}
\end{theorem}

As we mentioned above, this lower bound for the mass allows for volume degeneration.  In particular, we have 

\begin{corollary} \label{VolCor}  Let $M^4$ be a smooth, closed four-manifold, and $\{ \bg_i \}$ a sequence of Einstein metrics on $M^4$ with $Ric(\bg_i) = 3 \bg_i$. For each $i$, let $g_i = G_i^2 \bg_i$ be the corresponding AF metric defined by the Green's function $G_i$ for some choice of point $p_i \in M^4$, and $m(g_i)$ the ADM mass of $g_i$.  If $Vol(g_i) \to 0$, then $m(g_i) \to \infty$. 
\end{corollary}

If $M^4 = S^4$, then it is possible to say more.  Using compactness theory for Einstein metrics and a topological argument, Akutagawa-Endo-Seshadri showed that any sequence $\{ \bg_i \}$ of Einstein metrics on $S^4$ with $Ric(\bg_i) = 3 \bg_i$ and 
volume bounded below has a sequence that converges (mod diffeomorphisms) in $C^{\infty}$ to a smooth Einstein metric (see Proposition 1.4 and the arguments in Section 4 in \cite{AES}).  Combining this result with Corollary \ref{VolCor}, we have 

\begin{corollary}  \label{mhCor}  Let $\{ \bg_i \}$ be a sequence of Einstein metrics on $S^4$ with $Ric(\bg_i) = 3 \bg_i$. For each $i$, let $g_i = G_i^2 \bg_i$ be the corresponding AF metric defined by the Green's function $G_i$ for 
some choice of point $p_i \in M^4$, and $m(g_i)$ the ADM mass of $g_i$.  If there is a constant $\mathfrak{M}$ such that $m(g_i) \leq \mathfrak{M}$ for all $i$, then there is a subsequence $\{ \bg_{i_k} \}$, diffeomorphisms $\varphi_k : S^4 \rightarrow S^4$, and 
a smooth Einstein metric $\bg$ on $S^4$ such that $\{ \varphi_k^{\ast}\bg_{i_k} \}$ converges to $\bg$ in the $C^{\infty}$-topology. 
\end{corollary}

As indicated above, the volume estimate of \cite{Gur} combined with Theorem \ref{Thm1} implies the following inequality for the mass in terms of topological data:

\begin{theorem} \label{Thm3}  Under the assumptions of Theorem \ref{Thm1}, one of the following holds:  \vskip.1in

\noindent $(i)$  $(M^4,\bg)$ is homothetically isometric to $(S^4,g_0)$ or $(-\mathbb{CP}^2, g_{FS})$; or \vskip.1in

\noindent $(ii)$ If 
\begin{align*}
h(M) = \chi(M) + \frac{3}{2} \tau(M), 
\end{align*}
then the mass of $g = G^2\bg$ satisfies
\begin{align} \label{mt}
m(g) \geq \left[ 1 - \frac{1}{6} h(M) \right] + \frac{30}{h(M)} \left[ 1 - \frac{1}{6} h(M) \right]^2,
\end{align}
where $\chi(M)$ is the Euler characteristic and $\tau(M)$ the signature of $M$.  
\end{theorem}

We remark that since $(M^4,\bg)$ is a positive Einstein metric the Hitchin-Thorpe inequality \cite{Hitchin} implies $h(M) > 0$.  However, the lower bound in (\ref{mt}) 
becomes negative when $6 < h(M) < 8$.  

As immediate corollaries of Theorem \ref{Thm3} we have the following gap results for the mass:

\begin{corollary} \label{S4Cor}  Let $\bg$ be a positive Einstein metric on $S^4$, normalized so that $Ric(\bg) = 3\bg$.  Let $G$ denote the Green's function as above, and $g = G^2 \bg$.  If
\begin{align} \label{mgS4}
m(g) < \frac{22}{3},
\end{align}
then $\bg$ is isometric to the round metric.
\end{corollary}

\begin{corollary} \label{CP2Cor}  Let $\bg$ be a positive Einstein metric on $-\mathbb{CP}^2$, normalized so that $Ric(\bg) = 3 \bg$.  Let $G$ denote the Green's function as above, and $g = G^2 \bg$.  If
\begin{align} \label{mgCP2}
m(g) < 12,
\end{align}
then $\bg$ is homothetically isometric to $g_{FS}$.
\end{corollary}

\begin{remark}  \label{Rmk1}  If we conformally blow up $g_{FS}$ by the Green's function, then the resulting metric is known as the Burns metric, and the mass in independent of the choice of base point.  With our normalization the mass of Burns metric is $m = 1$.
\end{remark}

\medskip 

As the first of two final observations, we show that the mass can be estimated from below by the minimum of the Green's function:  

\begin{theorem} \label{MinGThm}   Under the assumptions of Theorem \ref{Thm1}, if the mass of $g = G^2 \bg$ satisfies $m(g) > 3$, then
\begin{align} \label{massmin}
m(g) \geq 4 \min G - 1. 
\end{align} 
\end{theorem}

\medskip 

The assumption  $m(g) > 3$ in the latter theorem appears to be just technical.  If $M = S^4$ and $\bg$ is not round, then by Corollary \ref{S4Cor} the condition automatically holds.  
Similarly, if $M^4 = \mathbb{CP}^2$ and $\bg$ is not (homothetically) $g_{FS}$, then  Corollary \ref{CP2Cor} also implies that $m(g) > 3$.   

\medskip 

In \cite{CJVXY} it is conjectured that any positive Einstein metric $\bg$ on a closed $n$-manifold $M$, normalized so that $Ric(\bg) = (n-1)\bg$, has a lower bound for the diameter 
which only depends on the dimension.  Using Theorem \ref{MinGThm}, we can give a lower bound for the diameter of any positive Einstein metric $\bg$ with $m(g) > 3$ in terms of the mass: 

\begin{theorem} \label{DMassThm}   Under the assumptions of Theorem \ref{Thm1}, if the mass of $g = G^2 \bg$ satisfies $m(g) > 3$, then the diameter
of $\bg$ satisfies 
\begin{align} \label{Dmass1} 
\mbox{diam}(\bg) \geq 2 \arctan \dfrac{1}{\sqrt{ 4 \min G - 1 }}.
\end{align}
Consequently, by Theorem \ref{MinGThm},
\begin{align} \label{Dmass2}
\mbox{diam}(\bg) \geq 2 \arctan \dfrac{1}{\sqrt{m(g)}}. 
\end{align} 
\end{theorem}

Again, if $\bg$ is a positive Einstein metric on $S^4$, then either $\bg$ is round (and equality holds in both (\ref{Dmass1}) and (\ref{Dmass2})), or $m(g) > 3$.  In particular, the conclusions hold for any positive 
Einstein metric on $S^4$ with our normalization. 

\medskip

\subsection{Organization of the paper}  The proof of Theorem \ref{Thm1} requires a slight refinement of the expansion of the Green's function in normal 
coordinates proved by Viaclovsky in \cite{JeffV}.  This is done in Section \ref{SecG}.  In Section \ref{Harnack} we prove a key gradient estimate for the Green's function.  In Sections \ref{SecThm1} - \ref{SecThm3} 
we give the proofs of Theorems \ref{Thm1} - \ref{Thm3} via the latter estimates, integration by parts and the Chern-Gauss-Bonnet theorem.  Finally, in Section \ref{minSec} we give the proofs of Theorems \ref{MinGThm} and \ref{DMassThm}, which are based 
on the maximum principle and a comparison between the Green's function and the distance from the set where 
it attains minimum.


\medskip 

\subsection{Acknowledgements}  This work was initiated during a visit of the first author to the Scuola Normale Superiore under the partial support of a Simons Foundation Fellowship in Mathematics, award
923208.  This author would like to thank the Scuola for its support and hospitality.  

The second author is supported by  the PRIN Project 2022AKNSE4 {\em Variational and Analytical aspects of Geometric PDE} and 
by the project {\em Geometric problems with loss of compactness} from Scuola Normale Superiore. He is also 
member of \emph{Gruppo Nazionale per l'Analisi Matematica, la Probabilità e le loro Applicazioni} (GNAMPA), as part of the Istituto Nazionale di Alta Matematica.

\medskip 

After the completion of this manuscript, we learned about recent results of Cosmin Manea which are related to some of the estimates in our paper, although with different proofs. In particular, using suitable monotonicity formulas and comparison principles, he obtained a version of  Proposition 3.3 which holds in all dimensions $n \geq 3$.

\medskip

\section{The Green's function} \label{SecG} 

As in the Introduction, $(M^4,\bg)$ is a closed, oriented, four-dimensional Einstein manifold with positive Einstein constant, normalized so that $Ric(\bg) = 3 \bg$, hence with scalar curvature $\overline{R} = 12$.  We fix a point $0 \in M^4$, and let $G$ denote the Green's function with respect to the conformal laplacian $L_{\bg} := \Delta_{\bg} - \frac{1}{6} \overline{R} = \Delta_{\bg} - 2,$ with pole at $0$.

By Proposition 2.1 of \cite{JeffV}, $G$ has an expansion in normal coordinates $\{ z^i \}$ centered at $0$ of the form
\begin{align} \label{Gexp}
G =  |z|^{-2} + A + a_k z^k + O_{1,\alpha}(|z|^{1+ \alpha})
\end{align}
for some $\alpha > 0$.  Here, we are using the convention that $f = O_{1,\alpha}(|z|^{\beta})$ means that $f \in C^{1,\alpha}$ with $f = O(|z|^{\beta})$, $|\partial f| = O(|z|^{\beta - 1})$.

Let $g = G^2 \bg$.  As we observed above, $g$ is scalar flat and asymptotically flat of order $2$.  By Proposition 2.2 of \cite{JeffV}, the mass of $g$ is related to the constant term in the expansion (\ref{Gexp}) by
\begin{align} \label{mA}
m(g) = 12 A - 1.
\end{align}

Using conformal covariance of $L$ and the fact that $L_{\bg} G = 0$ on $\widehat{M} = M \setminus \{ 0 \}$, it follows that $G$ satisfies the
following PDE with respect to $g$:
\begin{align} \label{DG}
\Delta G = 2 G^{-1} \left( 1 + |\nabla G|^2 \right),
\end{align}
where the laplacian is with respect to $g$.  Moreover, if $P$ denotes the Schouten tensor of $g$, i.e., 
\begin{align*}
P = \frac{1}{2} \left( Ric - \frac{1}{6}R g \right), 
\end{align*}
then
\begin{align} \label{PGg}
P = - \dfrac{\mathring{\nabla}^2 G}{G},
\end{align}
where $\mathring{\nabla}^2 G$ is the trace-free Hessian of $G$ (with respect to $g$).  \medskip

\noindent {\bf Note:}  From now on, we adopt the convention that quantities (curvature, covariant derivatives, etc.) computed with respect to $g$ will have no subscript, while quantities with respect to $\bg$ will be barred. 

\medskip

\subsection{Improved regularity of the Green's function}

As above, let $\{ z^i \}$ be $\bg$-normal coordinates centered at $0$.  If
\begin{align*}
x^i = \frac{z^i}{|z|^2}
\end{align*}
are inverted coordinates, then the expansion (\ref{Gexp}) with respect to these coordinates is
\begin{align} \label{Gx}
G = |x|^2 + A + \dfrac{a_k x^k}{|x|^2} + O_{1,\alpha}(|x|^{-1 - \alpha}),
\end{align}
where we now adopt the convention that $f = O_{k,\alpha}(|x|^{\beta})$ means $f \in C^{k,\alpha}_{\beta}(\widehat{M})$, the standard weighted H\"older space of 
functions (see \cite{LP}, Section 9).    

We will also need an expansion of the metric $g = G^2 \bg$.  We begin by recalling the expansion of $\bg$ in normal coordinates:
\begin{align*}
\bg_{ij} = \delta_{ij} - \frac{1}{3} \bar{R}_{i a j b}(0)z^a z^b - \frac{1}{6} \bar{\nabla}_c \bar{R}_{i a j b}(0)z^a z^b z^c + O(|z|^4).
\end{align*}
Using this and (\ref{Gx}), we can write the expansion of $g = G^2 \bg$ (cf.  (2.14) and (2.15) of \cite{JeffV}; but here we keep track of higher order errors):
\begin{align} \label{gexp} \begin{split}
g_{ij} &= \delta_{ij} - \frac{1}{3} \bar{W}_{iajb}(0)\frac{x^a x^b}{|x|^4} + \frac{1}{3} \frac{x^i x^j}{|x|^4} + (2A - \frac{1}{3})\frac{1}{|x|^2}\delta_{ij} \\
& \ \ \ \ \ \ + 2 \frac{a_k x^k}{|x|^4} \delta_{ij} - \frac{1}{6} \bar{\nabla}_c \bar{W}_{i a j b}(0)\frac{x^a x^b x^c}{|x|^6} + O_{1,\alpha}(|x|^{-3-\alpha}).
\end{split}
\end{align}
By the usual expansion of the inverse, it follows that
\begin{align} \label{ginvexp} \begin{split}
g^{ij} &= \delta_{ij} + \frac{1}{3} \bar{W}_{iajb}(0)\frac{x^a x^b}{|x|^4} - \frac{1}{3} \frac{x^i x^j}{|x|^4} - (2A - \frac{1}{3})\frac{1}{|x|^2}\delta_{ij} \\
& \ \ \ \ \ \ - 2 \frac{a_k x^k}{|x|^4} \delta_{ij} + \frac{1}{6} \bar{\nabla}_c \bar{W}_{i a j b}(0)\frac{x^a x^b x^c}{|x|^6} + O_{1,\alpha}(|x|^{-3-\alpha}).
\end{split}
\end{align}

The reason for carrying out more detailed expansions of the metric and its inverse is that in subsequent sections we will need the following refinement of (\ref{Gx}):

\begin{proposition} \label{GxProp}  In inverted normal coordinates,
\begin{align} \label{Gx2}
G = |x|^2 + A + \dfrac{a_k x^k}{|x|^2} + O_{2,\alpha}(|x|^{-1-\alpha}),
\end{align}
for some $\alpha \in (0,1)$. 
\end{proposition}

\begin{proof}  The improved regularity of the remainder will follow from the fact that $G$ satisfies an elliptic equation with respect to the AF metric $g$.   Recall from (\ref{DG}) that
\begin{align} \label{DG22}
\Delta G = 2 G^{-1} \left( 1 + |\nabla G|^2 \right).
\end{align}

\begin{lemma}  \label{DGxLemma}  The right-hand side of (\ref{DG22}) has an expansion of the form
\begin{align} \label{DG2x}
2 G^{-1} \left( 1 + |\nabla G|^2 \right) = 8 - 2m \frac{1}{|x|^2} - 32 \frac{a_k x^k}{|x|^4} + O_{0,\alpha}(|x|^{-3-\alpha}).
\end{align}
\end{lemma}

\begin{proof}  We first observe that by (\ref{Gx}),
\begin{align} \label{pxG}
\partial_j G = 2 x^j + \frac{a_j}{|x|^2} - 2 \frac{ a_k x^k x^j }{|x|^4} + O_{0,\alpha}(|x|^{-2-\alpha}),
\end{align}
where $\partial_j = \frac{\partial}{\partial x^j}$.  Then using the expansion of $g^{-1}$ above we have
\begin{align} \label{grad2G} \begin{split}
|\nabla & G|^2 = g^{ij} \partial_i G \partial_j G \\
&= \Big(  \delta_{ij} + \frac{1}{3} \bar{W}_{iajb}(0)\frac{x^a x^b}{|x|^4} - \frac{1}{3} \frac{x^i x^j}{|x|^4} - (2A - \frac{1}{3})\frac{1}{|x|^2}\delta_{ij} - 2 \frac{a_k x^k}{|x|^4} \delta_{ij} \\
& \ \ \ + \frac{1}{6} \bar{\nabla}_c \bar{W}_{i a j b}(0)\frac{x^a x^b x^c}{|x|^6} + O_{1,\alpha}(|x|^{-3-\alpha}) \Big) \left( 2 x^i + \frac{a_i}{|x|^2} - 2 \frac{ a_k x^k x^i }{|x|^4} + O_{0,\alpha}(|x|^{-2-\alpha}) \right) \\
& \ \ \times \left( 2 x^j + \frac{a_j}{|x|^2} - 2 \frac{ a_{\ell} x^{\ell} x^j }{|x|^4} + O_{0,\alpha}(|x|^{-2-\alpha}) \right) \\
&= \Big(  \delta_{ij} + \frac{1}{3} \bar{W}_{iajb}(0)\frac{x^a x^b}{|x|^4} - \frac{1}{3} \frac{x^i x^j}{|x|^4} - (2A - \frac{1}{3})\frac{1}{|x|^2}\delta_{ij} - 2 \frac{a_k x^k}{|x|^4} \delta_{ij} \\
& \ \ \ + \frac{1}{6} \bar{\nabla}_c \bar{W}_{i a j b}(0)\frac{x^a x^b x^c}{|x|^6} + O_{1,\alpha}(|x|^{-3-\alpha}) \Big) \Big( 4 x^i x^j + 2\frac{a_i x^j + a_j x^i}{|x|^2} - 8 \frac{ a_k x^k x^i x^j }{|x|^4} + O_{0,\alpha}(|x|^{-1-\alpha}) \Big) \\
&= 4 |x|^2 - 8A - 12 \frac{a_k x^k}{|x|^2} + O_{0,\alpha}(|x|^{-1-\alpha}).
\end{split}
\end{align}

By (\ref{Gx}),
\begin{align} \label{Gmx}
G^{-1} = |x|^{-2}  - A\frac{1}{|x|^4} - \frac{a_k x^k}{|x|^6} + O_{1,\alpha}(|x|^{-5-\alpha}).
\end{align}
Therefore, recalling (\ref{mA}), 
\begin{align*}
2 G^{-1} \left( 1 + |\nabla G|^2 \right) &= 2 \left( |x|^{-2}  - A\frac{1}{|x|^4} - \frac{a_k x^k}{|x|^6} + O_{1,\alpha}(|x|^{-5-\alpha})  \right) \\
& \ \ \ \ \times  \left(1 +  4 |x|^2 - 8A - 12 \frac{a_k x^k}{|x|^2} + O_{0,\alpha}(|x|^{-1-\alpha}) \right) \\
&= 8 - 2m \frac{1}{|x|^2} - 32 \frac{a_k x^k}{|x|^4} + O_{0,\alpha}(|x|^{-3-\alpha}).
\end{align*}
\end{proof}

Near infinity define
\begin{align} \label{Gz}
G_0 = |x|^2 + A + \dfrac{a_k x^k}{|x|^2}.
\end{align}
After multiplying by a cut-off function we may assume $G_0$ is globally defined.

\begin{lemma} \label{DGxLemma2}  Near infinity, $\Delta G_0$ has an expansion of the form
\begin{align} \label{Gzinf}
\Delta G_0 = 8 - 2m \frac{1}{|x|^2} - 32 \frac{a_k x^k}{|x|^4} + O_{0,\alpha}(|x|^{-3-\alpha}).
\end{align}
\end{lemma}

\begin{proof}  If $\Gamma_{ij}^k$ are the Christoffel symbols of $g$ with respect to inverted coordinates, then near infinity
\begin{align} \label{DG2} \begin{split}
\Delta G_0 &= g^{ij} \left( \partial_i \partial_j G_0 - \Gamma_{ij}^k \partial_k G_0 \right) \\
&= g^{ij}  \partial_i \partial_j G_0 - g^{ij} \Gamma_{ij}^k \partial_k G_0.
\end{split}
\end{align}

By the definition of $G_0$,
\begin{align*}
\partial_j G_0 &= 2 x^j + \frac{a_j}{|x|^2} - 2 \frac{a_k x^k x^j}{|x|^4}, \\
\partial_i \partial_j G_0 &= 2 \delta_{ij} - 2 \frac{(a_i x^j + a_j x^i)}{|x|^4} - 2 \frac{a_k x^k}{|x|^4}\delta_{ij} + 8 \frac{a_k x^k x^i x^j }{|x|^6}.
\end{align*}
Hence from (\ref{ginvexp}) 
\begin{align} \label{G2p}  \begin{split}
g^{ij}  \partial_i \partial_j G_0 &= \Big( \delta_{ij} + \frac{1}{3} \bar{W}_{iajb}(0)\frac{x^a x^b}{|x|^4} - \frac{1}{3} \frac{x^i x^j}{|x|^4} - (2A - \frac{1}{3})\frac{1}{|x|^2}\delta_{ij} - 2 \frac{a_k x^k}{|x|^4} \delta_{ij} \\
& \ \ + \frac{1}{6} \bar{\nabla}_c \bar{W}_{i a j b}(0)\frac{x^a x^b x^c}{|x|^6} + O_{1,\alpha}(|x|^{-3-\alpha}) \Big) \Big(  2 \delta_{ij} - 2 \frac{(a_i x^j + a_j x^i)}{|x|^4} \\
& \ \ \  - 2 \frac{a_k x^k}{|x|^4}\delta_{ij} + 8 \frac{a_k x^k x^i x^j }{|x|^6} \Big) \\
&= 8 + (2 - 16A) \frac{1}{|x|^2} - 20 \frac{a_k x^k}{|x|^4} + O_{0,\alpha}(|x|^{-3-\alpha}).
\end{split}
\end{align}

By the expansions (\ref{gexp}) and (\ref{ginvexp}) one can compute expansions for the Christoffel symbols with respect to $g$ to find
\begin{align} \label{Gamma}
g^{ij} \Gamma_{ij}^k =  4 A \frac{x^k}{|x|^4} - 2 \frac{a_k}{|x|^4} + 8 \frac{a_{\ell} x^{\ell} x^k}{|x|^6} + O_{\alpha}(|x|^{-4-\alpha}),
\end{align}
hence
\begin{align} \label{GdG} \begin{split}
- g^{ij} & \Gamma_{ij}^k \partial_k G_0  \\
&= - \left( 4 A \frac{x^k}{|x|^4} - 2 \frac{a_k}{|x|^4} + 8 \frac{a_{\ell} x^{\ell} x^k}{|x|^6} + O_{\alpha}(|x|^{-4-\alpha}) \right) \left( 2 x^k + \frac{a_k}{|x|^2} - 2 \frac{a_{m} x^m x^k}{|x|^4} \right)  \\
&= - 8A \frac{1}{|x|^2} - 12 \frac{a_k x^k}{|x|^4} + O_{\alpha}(|x|^{-3-\alpha}).
\end{split}
\end{align}
Combining (\ref{DG2}), (\ref{G2p}), and (\ref{GdG}), we get (\ref{Gzinf}).
\end{proof}

To complete the proof of Proposition \ref{GxProp}, we first observe that by (\ref{Gx}) and the definition of $G_0$ we have $G - G_0 \in C^{1,\alpha}_{-1- \alpha}$.  Also, by Lemmas \ref{DGxLemma} and \ref{DGxLemma2},
\begin{align*}
\Delta \left( G - G_0 \right) &= \Delta G - \Delta G_0 \\
&= 2 G^{-1} \left( 1 + |\nabla G|^2 \right) - \Delta G_0 \\
&= O_{\alpha}(|x|^{-3-\alpha}).
\end{align*}
By elliptic regularity in asymptotically flat spaces (see \cite{LP}, Theorem 9.2), it follows that $G - G_0 \in  C^{2,\alpha}_{-1-\alpha}$, hence (\ref{Gx2}) holds.    \end{proof} \medskip

The following corollary of Proposition \ref{GxProp} will be needed in the next section:  \medskip

\begin{corollary} \label{DGexpCor}  In inverted normal coordinates, we have
\begin{align} \label{DGalpha}
2 G^{-1} \left( 1 + |\nabla G|^2 \right) = 8 - 2m \frac{1}{|x|^2} - 32 \frac{a_k x^k}{|x|^4} + O_{1,\alpha}(|x|^{-3-\alpha}).
\end{align}
\end{corollary}

\begin{proof} The proof is the same as the proof of Lemma \ref{DGxLemma}; the only difference is that we use $G \in  C^{2,\alpha}_{-1-\alpha}$, hence the improved regularity of the remainder.
\end{proof}

\bigskip

\section{A gradient estimate}  \label{Harnack}

Let
\begin{align} \label{Wdef}
F = \dfrac{1 + |\nabla G|^2 - 4G}{G},
\end{align}
which will play a fundamental role in our estimates. From (\ref{DG}) it follows that 
\begin{align} \label{DFG}
\Delta G = 2 F + 8.
\end{align}

\begin{proposition}  \label{Prop1} $(i)$  On $\widehat{M} = M \setminus \{ 0 \}$,
\begin{align} \label{PG}
\Delta F = 2 G^{-1} |\mathring{\nabla}^2 G|^2.
\end{align}

\smallskip

\noindent $(ii)$  In inverted normal coordinates, $F$ has an expansion of the form
\begin{align} \label{Fexp}
F = - m \frac{1}{|x|^2} - 16 \frac{a_k x^k}{|x|^4} + O_{1,\alpha}(|x|^{-3-\alpha}).
\end{align}
\end{proposition}

\begin{proof} $(i)$  We begin with a lemma:

\begin{lemma} \label{dFLemma}  On $\widehat{M} = M \setminus \{ 0 \}$,
\begin{align} \label{dFID}
\nabla_j F = - 2 P_{jk} \nabla_k G,
\end{align}
where $P$ is the Schouten tensor of $g$.  
\end{lemma}

\begin{proof} By the definition of $F$,
\begin{align} \label{dF1} \begin{split}
\nabla_j F &= \nabla_j \left( G^{-1} \left( 1 + |\nabla G|^2 \right) - 4 \right) \\
&= - G^{-2} \left( 1 + |\nabla G|^2 \right) \nabla_j G + G^{-1} \nabla_j |\nabla G|^2.
\end{split}
\end{align}
Also, by (\ref{PGg})
\begin{align}
\label{dF2} \begin{split}
P_{jk}\nabla_k G &= - G^{-1} \left(\mathring{\nabla}^2 G\right)_{jk}\nabla_k G \\
&= - G^{-1} \left(\nabla_j \nabla_k G - \frac{1}{4} \left( \Delta G \right) g_{jk} \right)\nabla_k G \\
&= - G^{-1} \nabla_j \nabla_k G \nabla_k G  + \frac{1}{4} G^{-1} \left( \Delta G \right) \nabla_j G \\
&= - \frac{1}{2} G^{-1} \nabla_j |\nabla G|^2 + \frac{1}{4} G^{-1} \left( \Delta G \right) \nabla_j G \\
&= - \frac{1}{2} G^{-1} \nabla_j |\nabla G|^2 + \frac{1}{2} G^{-2} \left( 1 + |\nabla G|^2 \right) \nabla_j G,
\end{split}
\end{align}
where the last line follows from (\ref{DG}).  Comparing with (\ref{dF1}), we get (\ref{dFID}).
\end{proof}

\medskip

Since $g$ is scalar-flat, the second Bianchi implies that the Schouten tensor is divergence-free.  Therefore, taking the divergence of (\ref{dFID}) and using (\ref{PGg}) we find
\begin{align*}
\Delta F &= - 2 \nabla^j \left(  P_{jk} \nabla_k G \right) \\
&= - 2  P_{jk} \nabla_j \nabla_k G  \\
&= - 2 \left( - G^{-1} \mathring{\nabla}^2 G \right)_{jk} \nabla_j \nabla_k G \\
&= 2 G^{-1} |\mathring{\nabla}^2 G|^2,
\end{align*}
which proves (\ref{PG}).
\smallskip

\noindent $(ii)$  The expansion (\ref{Fexp}) follows immediately from Corollary \ref{DGexpCor}.
\end{proof}

\medskip

\begin{proposition} \label{AsFProp}  $F \leq 0$ on $\widehat{M}$, and equality is only achieved when $g$ is the flat metric.  In particular, we have the (sharp) gradient estimate
\begin{align} \label{gradG}
|\nabla G|^2 \leq 4 G - 1.
\end{align}
\end{proposition}

\begin{proof} Assume inverted coordinates are defined on $\{ |x| > \rho \}$ for $\rho > \rho_0 >> 1$.  Proposition \ref{Prop1} implies that $F$ is subharmonic on $M_{\rho} = M \setminus \{ |x| > \rho \}$, for any $\rho > \rho_0$.  Therefore, $F$ attains its maximum in $M_{\rho}$ on $\{ |x| = \rho \}$.  Using the expansion of $F$ in (\ref{Fexp}), on $\{ |x| = \rho \}$ we have
\begin{align} \label{Fr}
F = - m(g) \rho^{-2} + O(\rho^{-3}),
\end{align}
hence $F \rightarrow 0$ as $|x| \to \infty$, and it follows $F \leq 0$ on $\widehat{M}$.

If $F = 0$ at some point then by the strong maximum principle $F \equiv 0$.  By part $(i)$ of Proposition \ref{Prop1} this implies $\mathring{\nabla}^2 G \equiv 0$, and from (\ref{PGg}) we see that $g$ is Ricci-flat.  It then follows from Proposition 2 of \cite{Schoen} that $g$ is flat.
\end{proof}

\medskip

\begin{remark} \label{GroundRemark}   For the flat model on $\mathbb{R}^4$, $G$ is given by
\begin{align*}
G = |x|^2 + \frac{1}{4},
\end{align*}
and obviously $|\nabla_0 G|^2 = 4 G - 1$.
\end{remark}

\bigskip

\section{The proof of Theorem \ref{Thm1}}   \label{SecThm1}

\begin{proof}  By Proposition \ref{Prop1},
\begin{align} \label{mv1}  \begin{split}
\frac{1}{2} \Delta F^2 &=  F \Delta F + |\nabla F|^2 \\
&= \left( G^{-1} \left( 1 + |\nabla G|^2 \right) - 4 \right) \Delta F + |\nabla F|^2 \\
&= G^{-1} \Delta F + \frac{|\nabla G|^2}{G} \Delta F - 4 \Delta F + |\nabla F|^2 \\
&= G^{-1} \Delta F + 2 G^{-2} |\mathring{\nabla}^2 G|^2 |\nabla G|^2 - 4 \Delta F + |\nabla F|^2,
\end{split}
\end{align}
hence
\begin{align} \label{mv2}
\frac{1}{2} \Delta \left( F^2 + 8 F \right) = G^{-1} \Delta F + 2 G^{-2} |\mathring{\nabla}^2 G|^2 |\nabla G|^2 + |\nabla F|^2.
\end{align}
For $\rho >> 1$ sufficiently large, integrate the above over $M_{\rho} = M \setminus \{ |x| > \rho \}$ and apply the divergence theorem on the left-hand side:
\begin{align} \label{mv3} \begin{split}
\frac{1}{2} \oint_{\{ |x| = \rho \}} \frac{\partial}{\partial \nu} \left( F^2 + 8 F \right) \, dS &= \frac{1}{2} \int_{M_{\rho}} \Delta \left( F^2 + 8 F \right) dv \\
& = \int_{M_{\rho}}  G^{-1} \Delta F \, dv + 2  \int_{M_{\rho}} G^{-2} |\mathring{\nabla}^2 G|^2 |\nabla G|^2 \, dv + \int_{M_{\rho}} |\nabla F|^2 \, dv.
\end{split}
\end{align}

Using the expansion of $F$ in (\ref{Fexp}) on $\{ |x| = \rho \}$ we have
\begin{align} \label{Fr2}
F^2 + 8 F = - 8 m(g) \rho^{-2} + O_{1,\alpha}(\rho^{-3})
\end{align}
for some $\alpha \in (0,1)$, hence
\begin{align} \label{mF}
\frac{1}{2} \oint_{\{ |x| = \rho \}} & \frac{\partial}{\partial \nu} \left( F^2 + 8 F \right) \, dS = 16 \pi^2 m(g) + o(1), \ \ \rho \to \infty,
\end{align}
and consequently
\begin{align} \label{mv4}
 16 \pi^2 m(g)= \int_{M_{\rho}}  G^{-1} \Delta F \, dv + 2  \int_{M_{\rho}} G^{-2} |\mathring{\nabla}^2 G|^2 |\nabla G|^2 \, dv + \int_{M_{\rho}} |\nabla F|^2 \, dv + o(1),
\end{align}
as $\rho \to \infty$.     

For the first term on the right-hand side of (\ref{mv4}), we integrate by parts:
\begin{align} \label{mv5}  \begin{split}
\int_{M_{\rho}}  G^{-1} \Delta F \, dv &= \int_{M_{\rho}}  F \Delta \left( G^{-1}\right)  \, dv + \oint_{\{ |x| = \rho \}} \left( G^{-1} \frac{\partial}{\partial \nu}F - F \frac{\partial}{\partial \nu}\left(G^{-1} \right) \right) \, dS \\
&= \int_{M_{\rho}}  F \left( - G^{-2} \Delta G + 2 G^{-3} |\nabla G|^2 \right)  \, dv \\
& \ \ \ \  + \oint_{\{ |x| = \rho \}} \left( G^{-1} \frac{\partial}{\partial \nu}F + FG^{-2}  \frac{\partial}{\partial \nu}G \right) \, dS \\
&= - 2 \int_{M_{\rho}}  F G^{-3} \, dv + \oint_{\{ |x| = \rho \}} \left( G^{-1} \frac{\partial}{\partial \nu}F + FG^{-2}  \frac{\partial}{\partial \nu}G \right) \, dS,
\end{split}
\end{align}
and the last line follows from the previous line by (\ref{DG}).   Using the expansions for $F$ and $G$, it is easy to see that
\begin{align} \label{mv6}
\oint_{\{ |x| = \rho \}} \left( G^{-1} \frac{\partial}{\partial \nu}F + FG^{-2}  \frac{\partial}{\partial \nu}G \right) \, dS = o(1), \ \ \rho \to \infty.
\end{align}
For the interior integral in the right-hand side of \eqref{mv5}, since the volume forms are related by $dv = dv_g = G^4 dv_{\bg}$,
\begin{align} \label{mv7}
- 2 \int_{M_{\rho}}  F G^{-3} \, dv = - 2 \int_{M_{\rho}}  F G \, dv_{\bg}.
\end{align}
Also, by the expansions of $F$ and $G$, $|FG|$ is bounded, so
\begin{align} \label{mv8}
- 2 \int_{M_{\rho}}  F G^{-3} \, dv = 2 \int_{M_{\rho}}  F G \, dv_{\bg} = - 2 \int_M F G \, dv_{\bg} + o(1), \ \ \rho \to \infty.
\end{align}
Combining (\ref{mv5}), (\ref{mv6}), and (\ref{mv8}), we obtain
\begin{align} \label{mv9}
\int_{M_{\rho}}  G^{-1} \Delta F \, dv = - 2 \int_M  F G \, dv_{\bg} + o(1).
\end{align}

\medskip

\begin{lemma}  \label{FGLemma}
\begin{align} \label{FGform}
- \int F G \, dv_{\bg} = 3 \left( \frac{8}{3} \pi^2  -  \overline{V}\right),
\end{align}
hence
\begin{align} \label{mv9a}
\int_{M_{\rho}}  G^{-1} \Delta F \, dv = 6 \left( \frac{8}{3} \pi^2  -  \overline{V}\right) + o(1), \ \ \rho \to \infty.
\end{align}
\end{lemma}

\begin{proof}

By the definition of $F$,
\begin{align} \label{FS4} \begin{split}
- \int FG\, dv_{\bg} &= - \int \left( 1 + |\nabla_g G|_g^2 - 4G \right) \, dv_{\bg} \\
&= - \int \left( 1 + G^{-2}|\nabla_{\bg} G|_{\bg}^2 - 4G \right) \, dv_{\bg} \\
&= - \overline{V} - \int G^{-2}|\nabla_{\bg} G|_{\bg}^2  \, dv_{\bg} + 4 \int G \, dv_{\bg}.
\end{split}
\end{align}

\begin{claim} \label{IntClaim} We have
\begin{align} \label{IG1}
\int G^{-2}|\nabla_{\bg} G|_{\bg}^2  \, dv_{\bg} = 2\overline{V},
\end{align}
\begin{align} \label{IG2}
\int G \, dv_{\bg} = 2 \pi^2.
\end{align}
\end{claim}

\begin{proof}  Let $B_r$ be a small geodesic ball (with respect to $\bg$) centered at $0$.  By the definition of $G$,
\begin{align} \label{GG}
\Delta_{\bg} G = 2 G
\end{align}
on $M_r = M \setminus B_r$.  Therefore, dividing by $G$ and integrating gives
\begin{align} \label{IGp} \begin{split}
\int_{M_r} 2 \, dv_{\bg} &= \int_{M_r} G^{-1} \Delta_{\bg} G \, dv_{\bg} \\
&= - \int_{M_r} \langle \nabla_{\bg} G^{-1} , \nabla_{\bg} G \rangle_{\bg} \, dv_{\bg} + \oint_{\partial B_r} G^{-1} \frac{\partial G}{\partial \mu} \, dS_{\bg} \\
&=\int_{M_r} G^{-2} |\nabla_{\bg} G|^2 \, dv_{\bg} + \oint_{\partial B_r} G^{-1} \frac{\partial G}{\partial \mu} \, dS_{\bg}, \\
\end{split}
\end{align}
where $\mu$ is the outward unit normal to $\partial B_r$ with respect to $\bg$.  By the asymptotics of $G$, it is easy to check that
\begin{align*}
\oint_{\partial B_r} G^{-1} \frac{\partial G}{\partial \mu} \, dS_{\bg} = o(1), \ \ r \to 0.
\end{align*}
Therefore, letting $r \to 0$ in (\ref{IGp}) gives (\ref{IG1}).

To prove (\ref{IG2}) we just integrate (\ref{GG}) over $M_r$:
\begin{align*}
2 \int_{M_r} G \, dv_{\bg} = \int_{M_r} \Delta_{\bg} G \, dv_{\bg} = \oint_{\partial B_r} \frac{\partial G}{\partial \mu} \, dS_{\bg}.
\end{align*}
Again, by the asymptotics of $G$ it is easy to check that
\begin{align*}
\oint_{\partial B_r} \frac{\partial G}{\partial \mu} \, dS_{\bg} = 4 \pi^2 + o(1), \ \ r \to 0,
\end{align*}
and (\ref{IG2}) follows.
\end{proof}

Combining (\ref{FS4}), and the formulas in Claim \ref{IntClaim}, we arrive at (\ref{FGform}) and (\ref{mv9a}).   \end{proof}

Theorem \ref{Thm2} now follows from (\ref{mv4}) and the result of Lemma \ref{FGLemma}. \end{proof}

\section{The proof of Theorem \ref{Thm2}}   \label{SecThm2} 

\begin{proof}  For $\rho >> 1$ sufficiently large, let $M_{\rho} = \widehat{M} \setminus \{ |x| > \rho \}$, and define
\begin{align} \label{Ir}
I_{\rho} = \int_{M_{\rho}} \langle \nabla F, \nabla G \rangle G^{-2} \, dv.
\end{align}
Applying Cauchy-Schwarz,
\begin{align} \label{rI2}
I_{\rho} \leq \left( \int_{M_{\rho}} |\nabla F|^2 \, dv \right)^{1/2} \left( \int_{M_{\rho}} |\nabla G|^2 G^{-4} \, dv \right)^{1/2}.
\end{align}
Since $dv = G^4 dv_{\bg}$,
\begin{align*}
\int_{M_{\rho}} |\nabla G|^2 G^{-4} \, dv = \int_{M_{\rho}} G^{-2}|\nabla_{\bg} G|_{\bg}^2  \, dv_{\bg}.
\end{align*}
Therefore, by Claim \ref{IntClaim} we have
\begin{align*}
\lim_{\rho \to \infty} \int_{M_{\rho}} |\nabla G|^2 G^{-4} \, dv = 2\overline{V},
\end{align*}
hence
\begin{align} \label{rI3}
\lim_{\rho \to \infty} I_{\rho} \leq  \left( \int_{\widehat{M}} |\nabla F|^2 \, dv \right)^{1/2} \left( 2 \overline{V} \right)^{1/2}.
\end{align}

\begin{lemma}
\begin{align} \label{dFmass}
\int_{\widehat{M}} |\nabla F|^2 \, dv \leq \frac{48}{5} \pi^2 m(g) - \frac{18}{5} \left( \frac{8}{3} \pi^2 - \overline{V} \right).
\end{align}
\end{lemma}

\begin{proof} Recall from Theorem \ref{Thm1} that 
\begin{align} \label{MV1star}
 16 \pi^2 m(g) = 6 \left( \frac{8}{3} \pi^2 - \overline{V} \right) + 2 \int_{\widehat{M}} |\mathring{\nabla}_g^2 G|^2 \dfrac{|\nabla G|^2}{G^2} \, dv_g + \int_{\widehat{M}} |\nabla_g F|^2 \, dv_g,
\end{align}

We begin by showing that
\begin{align} \label{KI}
2 \int_{\widehat{M}} |\mathring{\nabla}^2 G|^2 \dfrac{|\nabla G|^2}{G^2} \, dv \geq \frac{2}{3} \int_{\widehat{M}} |\nabla F|^2 \, dv.
\end{align}
To see this, note that by Lemma \ref{dFLemma},
\begin{align} \label{dF2}
|\nabla F|^2 = 4 P^2 (\nabla G, \nabla G). 
\end{align}

\begin{claim}  If $V$ is an $n$-dimensional real inner product space, and $A : V \times V \rightarrow \mathbb{R}$ is a trace-free, symmetric bilinear form, then for each $\xi \in V$
\begin{align} \label{sharp}
A^2(\xi,\xi) \leq \frac{n-1}{n}|A|^2 |\xi|^2.
\end{align}
\end{claim}

The claim follows from an elementary Lagrange multiplier argument, and the proof will be omitted.  Since $P$ is trace-free, (\ref{PGg}), (\ref{dF2}), and the claim
imply
\begin{align} \label{dF3}
|\nabla F|^2 \leq 3 |P|^2 |\nabla G|^2 = 3 G^{-2} |\mathring{\nabla}^2 G|^2 |\nabla G|^2,
\end{align}
and (\ref{KI}) follows.  Then substituting (\ref{KI}) into (\ref{MV1star}), we have
\begin{align} \label{MassFV}
 16 \pi^2 m(g) \geq 6 \left( \frac{8}{3} \pi^2 - \overline{V} \right) +  \frac{5}{3}\int_{\widehat{M}} |\nabla F|^2 \, dv,
\end{align}
and (\ref{dFmass}) follows.
\end{proof}

From (\ref{rI3}) and (\ref{dFmass}), we see that
\begin{align} \label{rI4}
\lim_{\rho \to \infty} I_{\rho} \leq  \left\{ \frac{96}{5} \pi^2 m(g) \overline{V} - \frac{36}{5} \left( \frac{8}{3} \pi^2 - \overline{V} \right) \overline{V} \right\}^{1/2},
\end{align}
and the expression in braces is non-negative thanks to (\ref{dFmass}).  
Returning to the definition of $I_{\rho}$ in (\ref{Ir}), we can also write
\begin{align} \label{rI5} \begin{split}
I_{\rho} &= \int_{M_{\rho}} \langle \nabla F, \nabla G \rangle G^{-2} \, dv \\
&= - \int_{M_{\rho}} \langle \nabla F, \nabla ( G^{-1} )\rangle \, dv.
\end{split}
\end{align}
Integrating by parts,
\begin{align} \label{rI6} \begin{split}
I_{\rho} &= - \int_{M_{\rho}} \langle \nabla F, \nabla ( G^{-1} )\rangle \, dv \\
&= \int_{M_{\rho}} G^{-1} \Delta F \, dv - \oint_{\{|x| = \rho\}} G^{-1} \frac{\partial}{\partial \nu} F \, dS.
\end{split}
\end{align}
Using the asymptotics of $F$ and $G$, it is easy to see that
\begin{align*}
\oint_{\{|x| = \rho\}} G^{-1} \frac{\partial}{\partial \nu} F \, dS = o(1), \ \ \rho \to \infty.
\end{align*}
Also, by Lemma \ref{FGLemma},
\begin{align} \label{rI7}
\int_{M_{\rho}}  G^{-1} \Delta F \, dv = 6 \left( \frac{8}{3} \pi^2  -  \overline{V}\right) + o(1), \ \ \rho \to \infty.
\end{align}
Therefore,
\begin{align} \label{rI8}
\lim_{\rho \to \infty} I_{\rho} = 6 \left( \frac{8}{3} \pi^2  -  \overline{V}\right).
\end{align}
If we compare this with (\ref{rI4}), we see that
\begin{align} \label{rI9}
6 \left( \frac{8}{3} \pi^2  -  \overline{V}\right) \leq \left\{ \frac{96}{5} \pi^2 m(g) \overline{V} - \frac{36}{5} \left( \frac{8}{3} \pi^2 - \overline{V} \right) \overline{V} \right\}^{1/2}.
\end{align}
Squaring both sides and rearranging terms, we arrive at (\ref{massgap}).
\end{proof}

\bigskip

\section{The proofs of Theorem \ref{Thm3} and its corollaries}  \label{SecThm3} 

 
\begin{proof}[Proof of Theorem \ref{Thm3}]  By Theorem A of \cite{Gur}, either $(M^4,\bg)$ is anti-self-dual, or the self-dual Weyl tensor satifies the bound
\begin{align} \label{L2W}
\int |W_{\bg}^{+}|^2 \, dv_{\bg} \geq \frac{4}{3} \pi^2 \left( 2 \chi(M) + 3 \tau(M) \right).
\end{align}
If $(M^4,\bg)$ is anti-self-dual, then by a result of Hitchin \cite{Hitchin} it is homothetically isometric to $(S^4,g_{round})$ or to $(-\mathbb{CP}^2, g_{FS})$.  Assume, therefore, that $(M^4, \bg)$ is not ASD.

Since $(M^4, \bg)$ is Einstein with $R_{\bg} = 12$, by the Chern-Gauss-Bonnet and signature formulas we have
\begin{align} \label{CGB}
2 \pi^2 \left( 2 \chi(M) + 3 \tau(M) \right) = \int |W_{\bg}^{+}|^2 \, dv_{\bg} + 3 \overline{V}.
\end{align}
Combining (\ref{L2W}) and (\ref{CGB}), we find that
\begin{align} \label{vb}
\overline{V} \leq \frac{2}{9} \pi^2 \left( 2 \chi(M) + 3 \tau(M) \right).
\end{align}
If we define 
\begin{align*}
\theta_{\bg} = \frac{3 \overline{V}}{8 \pi^2},
\end{align*}
then (\ref{vb}) implies
\begin{align} \label{tb}
\theta_{\bg} \leq \frac{1}{12} \left( 2 \chi(M) + 3 \tau(M) \right) := \theta_M.
\end{align}
Also, we can express the conclusion of Theorem \ref{Thm2} as
\begin{align*}
m(g) \geq \left( 1 -  \theta_{\bg}  \right) + 5  \theta^{-1}_{\bg} \left( 1 - \theta_{\bg} \right)^2.
\end{align*}
If $f : \mathbb{R}^{+} \rightarrow \mathbb{R}$ is given by
\begin{align*}
f(x) =  \left( 1 - x \right) + 5 x^{-1}  \left( 1 - x \right)^2,
\end{align*}
then $f$ is strictly decreasing for $0 < x < 1$.  Therefore, since $\theta_{\bg} \leq \theta_M$, it follows that
\begin{align} \label{miq}
m(g) \geq f(\theta_{\bg}) \geq f(\theta_M),
\end{align}
which gives (\ref{mt}).
\end{proof}

\medskip 

\subsection{The proof of Corollary \ref{S4Cor}}  Taking $\chi(M) = 2$ and $\tau(M) = 0$ we find $\theta_M = \frac{1}{3}$, and (\ref{mt}) implies
\begin{align*}
m(g) \geq f(\frac{1}{3}) = \frac{22}{3}.
\end{align*}

\medskip

\subsection{The proof of Corollary \ref{CP2Cor}} Taking $\chi(M) = 3$ and $\tau(M) = -1$ we find $\theta_M = \frac{1}{4}$, and (\ref{mt}) implies
\begin{align*}
m(g) \geq f(\frac{1}{4}) = 12.
\end{align*}
 
\medskip 

\section{The proofs of Theorem \ref{MinGThm} and \ref{DMassThm}} \label{minSec} 

\begin{proof}[The proof of Theorem \ref{MinGThm}]   If $\bg = g_0$ is the round metric, then $m(g) = 0$ and $\min G = 1/4$, so the result follows.  Therefore, we assume $\bg$ is 
not round and $m(g) > 3$.   On $\widehat{M}$ define 
\begin{align} \label{Hdef}
H_{\epsilon} = F + \left( m + \epsilon \right) G^{-1}.  
\end{align}
Let $\{ x^i \}$ be inverted normal coordinates.   By (\ref{Gx}) and part $(ii)$ of Proposition \ref{Prop1},
\begin{align} \label{Hex}
H_{\epsilon}(x) = \dfrac{\epsilon}{|x|^2} + O_{1,\alpha}(|x|^{-3}), \ \ |x| \to \infty. 
\end{align} 
For $\rho >> 1$ large let $M_{\rho} = \widehat{M} \setminus \{ |x| > \rho \}$.  Then (\ref{Hex}) implies that once $\rho = \rho_0 >> 1$ is large enough, the normal derivative of $H_{\epsilon}$ on $\{ |x| = \rho_0 \}$ (with respect to the 
outward pointing unit normal) is positive; hence the maximum of $H_{\epsilon}$ in $M_{\rho_0}$ cannot be attained on the boundary.   Let $p$ be a point in $M_{\rho_0}$ at which $H_{\epsilon}$ attains it maximum.   

We claim that $\nabla G(p) = 0$.  To see this, assume $\nabla G(p) \neq 0$, and compute the laplacian of $H_{\epsilon}$ using part $(i)$ of Proposition \ref{Prop1} and (\ref{DG}):  
\begin{align} \label{DH2} \begin{split}
\Delta H_{\epsilon} &= \Delta F + \left( m + \epsilon \right) \Delta G^{-1} \\
&= 2 G^{-1} |\mathring{\nabla}^2 G|^2 - 2 \left( m + \epsilon \right) G^{-3}. 
\end{split}
\end{align}
Since $H_{\epsilon}$ attains its maximum at $p$, at $p$ we have 
\begin{align} \label{DH3} 
0 \geq  2 G^{-1} |\mathring{\nabla}^2 G|^2 - 2 \left( m + \epsilon \right) G^{-3}. 
\end{align}
In addition, $\nabla H_{\epsilon}(p) = 0$, which implies 
\begin{align} \label{DH4} 
0 = \nabla F - \left( m + \epsilon \right) G^{-2} \nabla G, 
\end{align} 
hence 
\begin{align} \label{DH5} 
\nabla F =  \left( m + \epsilon \right) G^{-2} \nabla G
\end{align} 
at $p$.  

By (\ref{dF3}), 
\begin{align} \label{DH6} \begin{split}
2 G^{-1} |\mathring{\nabla}^2 G|^2 &=  2 \left( G^{-2}|\mathring{\nabla}^2 G|^2 |\nabla G|^2 \right)\left( \dfrac{G}{|\nabla G|^2} \right) \\
&\geq \frac{2}{3} |\nabla F|^2 \dfrac{G}{|\nabla G|^2}, 
\end{split}
\end{align}
where we are using the fact that $\nabla G(p) \neq 0$ by assumption.  Using (\ref{DH5}) it follows that at $p$, 
\begin{align} \label{DH7} 
2 G^{-1} |\mathring{\nabla}^2 G|^2 \geq \frac{2}{3} \left( m + \epsilon \right)^2 G^{-3}. 
\end{align}
Combining this with (\ref{DH3}), we find 
\begin{align} \label{DH8} \begin{split}
0 &\geq \left\{ \frac{2}{3} \left( m + \epsilon \right)^2  - 2 \left( m + \epsilon \right) \right\} G^{-3} \\
&=  \left( m + \epsilon \right) \left\{ \frac{2}{3} \left( m + \epsilon \right)  - 2  \right\} G^{-3}.
\end{split}
\end{align}
Since $m(g) > 3$ the right-hand side of (\ref{DH8}) is positive, which is a contradiction.  We conclude that $\nabla G(p) = 0$.  

Let $q$ be a point at which $G$ attains its minimum.  Then since $\nabla G(p) = \nabla G(q) = 0$, 
\begin{align*}
H_{\epsilon}(p) &= \left(1 + m + \epsilon \right) G^{-1}(p) - 4  \\
&\leq \left(1 + m + \epsilon \right) \left( \min G \right)^{-1} - 4 \\
&=H_{\epsilon}(q). 
\end{align*}
Consequently, the maximum of $H_{\epsilon}$ in $M_{\rho_0}$ is attained at $q$, and 
\begin{align*}
\max_{M_{\rho_0}} H_{\epsilon} &= H_{\epsilon}(q) \\
&= \left(1 + m + \epsilon \right) \left( \min G \right)^{-1} - 4. 
\end{align*}
Also, 
\begin{align*}
\max_{M_{\rho_0}} H_{\epsilon} &\geq \max_{\{ |x| = \rho_0 \}} H_{\epsilon} \\
&= \epsilon \rho_0^{-2} + O(\rho_0^{-3}).
\end{align*}
Since this holds for all $\rho_0 >> 1$ sufficiently large, we conclude  
\begin{align} \label{maxHe}
\left(1 + m + \epsilon \right) \left( \min G \right)^{-1} - 4 \geq 0,
\end{align}
or 
\begin{align*}
m + \epsilon \geq 4 \min G - 1. 
\end{align*}
As this holds for all $\epsilon > 0$, the result follows. 
\end{proof}

\medskip 

\begin{proof}[The proof of Theorem \ref{DMassThm}]  By Proposition \ref{AsFProp},
\begin{align} \label{G1est}
|\nabla G|^2 \leq 4 G - 1 = 4 \left( G - \frac{1}{4} \right)
\end{align}
on $M^4 \setminus \{ 0 \}$.  Since $|\nabla G|^2 = G^{-2} |\overline{\nabla} G|_{\bg}^2$, where $\overline{\nabla}$ is the gradient of $G$ with respect to $\bg$, it follows that on $M \setminus \{ 0 \}$ we have 
\begin{align} \label{DH34}
\dfrac{|\overline{\nabla}G|_{\bg}}{2 G \sqrt{ G - \frac{1}{4} }} \leq 1.
\end{align}
Define $f : M^4 \setminus \{ 0 \} \rightarrow \mathbb{R}$ by 
\begin{align} \label{fdef}
f = 2  \arctan \left( 2 \left[ G - \frac{1}{4} \right]^{1/2} \right).
\end{align}
Then (\ref{DH34}) implies 
\begin{align} \label{DH35} 
|\overline{\nabla}f|_{\bg} \leq 1. 
\end{align}
Let $q \in M^4$ be a point at which $G(q) = \min G$.  Let $L = \mbox{dist}_{\bg}(q,0)$; then by Myers' theorem, $L \leq \pi$.  Let $\gamma : [0,L] \to M^4$ be a unit speed geodesic from $q$ to $0$.  Then for any $0 \leq \tau < L$, 
\begin{align} \label{DH36} \begin{split}
f\left( \gamma(\tau) \right) - f\left( \gamma(0) \right) &= \int_0^{\tau} \frac{d}{dt} f(\gamma(t)) \, dt \\
&= \int_0^{\tau} \left\langle \bar{\nabla}f, \frac{d\gamma}{dt} \right\rangle_{\bg} \, dt \\
&\leq \int_0^{\tau} \left|\bar{\nabla}f\right|_{\bg} \left|\frac{d\gamma}{dt}\right|_{\bg} \, dt \\
&\leq \tau. 
\end{split}
\end{align}
If we denote $q_{\tau} = \gamma(\tau)$, then 
\begin{align} \label{DH38}
2 \arctan \left( 2 \left[ G(q_{\tau}) - \frac{1}{4} \right]^{1/2} \right) - 2 \arctan \left( 2 \left[ \min G - \frac{1}{4} \right]^{1/2} \right)  \leq \tau,
\end{align}
and using the formula for the difference of the inverse tangent of two numbers we obtain 
\begin{align} \label{DH389}
2 \arctan \left( \dfrac{ 2 \left( \left[ G(q_{\tau}) - \frac{1}{4} \right]^{1/2} - \left[ \min G - \frac{1}{4}\right]^{1/2} \right)}{ 1 + 4  \left[ G(q_{\tau}) - \frac{1}{4} \right]^{1/2} \left[ \min G - \frac{1}{4} \right]^{1/2}} \right) \leq \tau. 
\end{align}
Dividing by two and taking the tangent, 
\begin{align} \label{DH390}
\dfrac{ 2 \left( \left[ G(q_{\tau}) - \frac{1}{4} \right]^{1/2} - \left[ \min G - \frac{1}{4}\right]^{1/2} \right)}{ 1 + 4  \left[ G(q_{\tau}) - \frac{1}{4} \right]^{1/2} \left[ \min G - \frac{1}{4} \right]^{1/2}} \leq \tan \frac{1}{2} \tau. 
\end{align}
Letting $\tau \to L$, since $G(q_{\tau}) \to \infty$ we get 
\begin{align} \label{DH391} 
\dfrac{1}{2 \left[ \min G - \frac{1}{4} \right]^{1/2}} \leq \tan \frac{1}{2} L \leq \tan \frac{1}{2} \mbox{diam}(\bg),
\end{align} 
which implies 
\begin{align} \label{DH392} 
\mbox{diam}(\bg) \geq 2 \arctan \dfrac{1}{2 \left[ \min G - \frac{1}{4} \right]^{1/2}} = 2 \arctan \dfrac{1}{\sqrt{ 4 \min G - 1 }}. 
\end{align}
Recall from Theorem \ref{MinGThm} that $m(g) \geq 4 \min G - 1$, hence 
\begin{align*}
\mbox{diam}(\bg) \geq 2 \arctan \dfrac{1}{\sqrt{m(g)}}. 
\end{align*}

\end{proof}

\vskip.5in

%
%


\begin{thebibliography}{s}
%
%
%
%




\bibitem{AES} K. Akutagawa, H. Endo, H. Seshadri, A gap theorem for positive Einstein metrics on the four-sphere, {\em Math. Ann.} {\bf 373} (2019), no. 3-4, 1329-1339.

\bibitem{Bartnik} R. Bartnik, The mass of an asymptotically flat manifold, {\em Comm. Pure Appl.Math.} {\bf 39} (1986) (5), 661-693.

\bibitem{Bohm} C. Bohm, Inhomogeneous Einstein metrics on low-dimensional spheres and other low-dimensional spaces, {\em Invent. Math.} {\bf 134} (1998), 145–176. 

\bibitem{BGK} C. P. Boyer, K. Galicki, J. Kollar, Einstein metrics on spheres, {\em Ann. of Math.} (2) {\bf 162} (2005), no. 1, 557–580. 

\bibitem{Besse}  Besse, A.L.: Einstein Manifolds. Springer, Berlin (1987)

\bibitem{BK}  J.-P. Bourguignon and H. Karcher, Curvature operators: pinching estimates and geometric examples, {\em Ann. Sci. Ecole Norm. Sup.} {\bf 11} (1978), 71–92.

\bibitem{Cheng}  S. Y. Cheng, Eigenvalue comparison theorems and its geometric applications, {\em Math. Z.} {\bf 143} (1975), 289-297.

\bibitem{CS} T. C. Collins, G. Szekelyhidi, Sasaki-Einstein metrics and K-stability, {\em Geom. Topol.} {\bf 23} (2019) (3), 1339–1413.

\bibitem{CJVXY} T. C. Collins, D. Jafferis, C. Vafa, K. Xu, S.-T. Yau,  On Upper Bounds in Dimension Gaps of CFT's, arXiv:2201.03660 [hep-th]

\bibitem{Gur} M. J. Gursky, Four-manifolds with $\delta W^{+}=0$ and Einstein constants of the sphere, {\em Math. Ann.} {\bf 318} (2000), no. 3, 417-431.

\bibitem{HL} H.-J. Hein, C. LeBrun, Mass in K\"ahler geometry, {\em Comm. Math. Phys.} {\bf 347} (2016), no. 1, 183–221.

\bibitem{Hitchin} N. Hitchin, On compact four-dimensional Einstein manifolds, {\em J. Diff. Geom.} {\bf 9} (1974), 435-442.

\bibitem{Jen} G. R. Jensen, Einstein metrics on principal fibre bundles, {\em J. Differential Geom.} {\bf 8} (1973), 599–614.

\bibitem{LP} J. M. Lee and T. H. Parker, The Yamabe problem, {\em Bull. Amer. Math. Soc. (N.S.)} {\bf 17} (1987), no. 1, 37–91.

\bibitem{NW} J. Nienhaus and M. Wink, Einstein metrics on the ten-sphere, arXiv:2303.04832v3 [math.DG]

\bibitem{Schoen} R. Schoen, Conformal deformation of a Riemannian metric to constant scalar curvature, {\em J. Differential Geom.}  {\bf 20} (1984), no.2, 479–495.

\bibitem{SY1}  R. S. Schoen, S.T. Yau, On the proof of the positive mass conjecture in general relativity, {\em Commun. Math. Phys.} {\bf 65} (1979) (1), 45-76.

\bibitem{SY2}  R. S. Schoen, S. T. Yau, Proof of the positive mass theorem. II, {\em Commun. Math. Phys.} {\bf 79} (1981) (2),
231-260.

\bibitem{Witten} E. Witten, A new proof of the positive energy theorem, {\em Commun. Math. Phys.} {\bf 80} (1981) (3), 381-402.


\bibitem{JeffV} J. A. Viaclovsky, The mass of the product of spheres, {\em Comm. Math. Phys.} {\bf 335}  (2015), no.1, 17-41.

\end{thebibliography}
\end{document}